\documentclass[11pt, a4paper, twoside, regno]{amsart}
\calclayout
\usepackage[T1]{fontenc}
\usepackage{color}
\usepackage[english]{babel}
\usepackage[mathscr]{euscript}
\usepackage{MnSymbol}
\usepackage[mathscr]{euscript}
\usepackage{amsthm}
\usepackage{thmtools}
\usepackage{stmaryrd}
\usepackage[utf8]{inputenc}
\usepackage[all]{xy}
\usepackage{indentfirst}
\usepackage{hyperref}
\usepackage{amsmath}
\usepackage{amsmath,amsfonts}
\usepackage{lipsum}
\usepackage[toc,page]{appendix}
\usepackage[pdftex]{graphicx}
\usepackage{tikz-cd}
\usepackage{amsthm}
\usepackage{ragged2e}
\usepackage{blindtext}
\usepackage{url}
\usepackage{booktabs} 
\usepackage{array} 
\usepackage{paralist} 
\usepackage{verbatim} 
\usepackage{subfig}
\usepackage{graphicx}			
\input xy
\xyoption{all}

\newtheorem{theorem}[subsection]{Theorem}
\newtheorem{Theorem}[theorem]{Theorem}

\newtheorem{Proposition}[theorem]{Proposition}
\newtheorem{Definition}[theorem]{Definition}
\newtheorem{Remark}[theorem]{Remark}

\newtheorem{Example}[theorem]{Example}

\DeclareMathOperator{\rank}{rank}

\DeclareMathOperator{\Fitt}{Fitt}

\DeclareMathOperator{\Ann}{Ann}

\DeclareMathOperator{\coker}{Coker}

\DeclareMathAlphabet{\mathpzc}{OT1}{pzc}{m}{it}

\title{The dual F-Signature of Veronese Rings}
\author[Vinicius Bou\c{c}a]{Vinicius Bou\c{c}a}
\address{$^{1}$  Instituto de Matem\'{a}tica,
	Universidade Federal do Rio de Janeiro,  Brazil. }
\email{vbouca@im.ufrj.br }
\author[Eliana Tolosa Villarreal]{Eliana Tolosa Villarreal}
\address{$^{2}$  Instituto de Matem\'{a}tica,
	Univertistà degli studi di Genova, Italy }
\email{etolosav@gmail.com}
\author[Kevin Vasconcellos]{Kevin Vasconcellos}
\address{$^{3}$  Instituto de Matem\'{a}tica,
	Universidade Federal do Rio de Janeiro,  Brazil. }
\email{kevin.vasconcellos@im.ufrj.br}
\date{}
\begin{document}
\maketitle
\begin{abstract}
In this paper we adress the question of I. Smirnov and K. Tucker on the dual F-signature of the Veronese subrings of polynomial rings in $n$ variables using methods of commutative algebra.
\end{abstract}
\thispagestyle{empty}
\section{Introduction}
Let $R$ be a complete $d$-dimensional reduced Noetherian local ring with prime characteristic $p > 0$ and perfect residue field $K = K^p$. For $e \in \mathbb{N}$, we can naturally identify the inclusion $R \subseteq R^{1/p^e}$ into the $p^e$-th roots of elements of $R$, with the $e$-th iterate of the Frobenius endomorphism. The behaviour of such endomorphism characterize some singularities, called the $F$-singularities. Among the classes of $F$-singularities, three hold primordial significance: $F$-purity, strong $F$-regularity and $F$-rationality. To investigate and classify $F$-singularities, several numerical invariants have been introduced since the past four decades.

Let us decompose $R^{1/p^e}$ as the direct sum of free $R$-modules and a non-free summand $M_e$ and
let $a_e$ denote the largest rank of the free $R$-module appearing in the decomposition. That is,
\begin{equation*}
    R^{1/p^e} = R^{\bigoplus a_e} \bigoplus M_e.
\end{equation*}
The number $a_e$ is called the \textit{e-th Frobenius splitting number} and it gives information on how the Frobenius endomorphism acts on $R$. In order to study the asymptotic behaviour of $a_e$, it was defined the 
$F$\textit{-signature} as the following limit
\begin{equation*}
    s(R) = \lim_{e \xrightarrow{}\infty} \frac{a_e}{p^{ed}}.
\end{equation*}

The $F$-signature was first implicitly mentioned in the work of K. Smith and M. Van der Bergh \cite{smith2002simplicity} in 1997 and then it was formally introduced and studied by C. Huneke and G. Leuschke \cite{huneke2002two} in 2002. Nonetheless, the existence of the limit was not clear until $2012$, when K. Tucker prove it in general \cite{Tucker_2012}.

This invariant carries interesting information about the singularities of $R$. In fact, if $R$ is a regular ring, $R^{1/p^e}$ is a free $R$-module of rank $p^{ed}$, meaning that the $F$-signature somehow measures how far is the ring $R$ to be regular. C. Huneke and G. Leuschke \cite{huneke2002two} proved that $s(R)\geq 1$ with equality if and only if $R$ is regular. Furthermore, I. Aberbach and G. Leuschke \cite{aberbach2002f} showed that $s(R)>0$ if and only if $R$ is strongly $F$-regular.

To study the relationship between the $F$- signature and $F$-rationality, A. Sannai \cite{sannai2015dual} expanded the definition of $F$-signature to encompass modules, introducing the \textit{dual $F$-signature} and defined as follows

$$s_{dual}(M) = \lim \sup_{q\to \infty}\frac{\max \{ N \hspace{2pt} | \hspace{2pt} \text{ there is a sujection } F_{*}^e \omega_R \twoheadrightarrow \omega_R^N\} }{\rank F_{*}^e \omega_R},$$
where $R$ is assumed Cohen-Macaulay and $\omega_R$ is its canonical module.

Sannai established that, for $F$-finite reduced Cohen-Macaulay local rings with characteristic $p>0$ and admitting canonical module $\omega_R$, the condition of $F$-rationality is unequivocally defined: $R$ is $F$-rational if and only if the dual $F$-signature of its canonical module is positive. 

However, calculating the $F$-signature and the dual $F$-signature is not trivial. The question is still open even for well studied rings.

Our work focuses on calculating the dual $F$-signature of the $d$-th Veronese subring $S^{(d)}$ of the polynomial ring $S = k[x_1,\dots,x_n]$, validating the suspicion presented by Smirnov and Tucker in  \cite{smirnov2023theory}. We state our main result next.

\begin{Theorem}
Let $k$ be a perfect field of prime characteristic $p > 0$ and $S^{(d)}$ the $d$-Veronese subring of $k[x_1,\dots,x_n]$. Then, the dual $F$-signature of $S^{(d)}$ is
\begin{gather*}
s_{dual}\left(S^{(d)}\right) = \frac{1}{d}\bigg\lceil\frac{d}{n}\bigg\rceil.
\end{gather*}       
\end{Theorem}\vspace{0.1cm}

The paper is organized as follows. In section 2 we recall some basic definitions about the dual $F$-signature of a module and stablish the notation used throughout the paper. In section 3, we explicitly give a decomposition of the module of $p^e$-roots canonical module of the Veronese rings $S^(d)$ as a direct sum of $S^(d)$-modules. IN section 4 we pause the discussion on the dual F-signature to prove an auxiliary result that is used  
Finally, in section 5, we sate again our main question and prove it by bounding above the F-signature by counting generators and bounding it below by exhibiting explicit maps between 

\section{Preliminaries}
Throughout what follows, $(R,\mathfrak{m},K)$ is a $d$-dimensional reduced Noetherian ring of prime characteristic $p>0$. We use the symbol $q$ to represent a varying power of $p$ in our notation. We set $\alpha(R) = \log_p[K\::\:K^p]$ and assume that $R$ is $F$-finite, which means that the Frobenius endomorphism is finite. Equivalently, considering $R^{1/q} = \{r^{1/q}\:\:;\:\:r \in R\}$ the ring of $q$-th roots of elements of $R$, $R$ is $F$-finite if $R^{1/q}$ is a finite $R$-module, which implies that $\alpha(R) < \infty$. In the following we present the definition of the $F$-signature of $R$.
\begin{Definition}
Let $(R,\mathfrak{m},k)$ be a ring as above. For each $q = p^e$, decompose $R^{1/q}$ as a direct sum of finite $R$-modules $R^{a_q}\oplus M_q$, where $M_q$ does not contain non-zero free direct summands. The $F$-signature of $R$ is
\begin{gather*}
s(R) = \lim_{q\to \infty}\frac{a_q}{q^{d + \alpha(R)}}.
\end{gather*}
\end{Definition}
\noindent For any positive integer $e$, we define the ring endomorphism $F^e$ through the composition of the Frobenius endomorphism applied $e$ times. Consequently, for an $R$-module $M$, this endomorphism induces on $M$ a new $R$-module structure on $M$, denoted as $F_{*}^eM$. Sannai \cite{sannai2015dual} extended the concept of $F$-signature for $R$-modules, introducing what he referred to as the dual $F$-signature.
\begin{Definition}
Let $(R,\mathfrak{m},k)$ be a ring as above and $M$ an $R$-module. For each $q$, let $b_q$ be the $F$-surjective number of $M$ defined by $$b_q = \max\{n \in \mathbb{N}\:\:;\:\:\exists\:F_{*}^e:\:M\longrightarrow M^n\:\:surjective\}.$$ We define the dual $F$-signature of $M$ by
\begin{gather*}
s_{dual}(M) = \lim \sup_{q\to \infty}\frac{b_q}{q^{d + \alpha(R)}}.
\end{gather*}
\end{Definition}
Let $K$ be a perfect field with a prime characteristic $p>0$. Consider $S = K[x_1, \dots, x_n]$ the polynomial ring over $K$ with $n$ indeterminates and equipped with the standard grading. Let's denote by $S_i$ the $i$-th homogeneous component of the polynomial ring $S$. This component is spanned over $K$ by all monomials of $S$ that possess a degree $i$. This leads us the following direct sum decomposition 
\begin{gather*}
S = \bigoplus_{i=0}^\infty S_i.
\end{gather*}
The $d$-th Veronese ring of $S$, commonly denoted as $S^{(d)}$, is the graded subring generated over $K$ by all monomials of degree $d$, that is
\begin{gather*}
S^{(d)} = \bigoplus\limits_{i=0}^\infty S_{id}.
\end{gather*}
Observe that $S$ can be regarded as a finite module over $S^{(d)}$. With this structure, $S$ decomposes into a direct sum of $S_d$-modules
\begin{gather*}
S = \bigoplus\limits_{j = 0}^{d-1} S_{[j]},
\end{gather*}
where
\begin{gather*}
S_{[j]} = \bigoplus\limits_{i=0}^\infty S_{j + id}
\end{gather*}
for $j=0,1,\dots,d-1$. Notice that we can think of $S_{[j]}$ as the polynomials with degree $j$ modulo $d$.
Lastly, the superscript $^{1/p^e}$ shall symbolize the ring (or module) resulting from taking $p$-th roots. 
\begin{Remark}
    Note that, if $S = K[x_1,\dots,x_n]$ is the polynomial ring over $K$ with $n$ indeterminates over a perfect field of prime characteristic $p>0$, then its $d$-Veronese subring $S^{(d)}$ is a reduced $F$-finite ring. 
\end{Remark}
\section{The Structure of the Canonical Module of Veronese Rings}
		
In this section we will explain the structure of the canonical module of the Veronese ring $S^{(d)}$. Recall that the canonical module of the polynomial ring $S$ is $S(-n)$. Hence the canonical module $\omega_{S^{(d)}}$ is given by
		$$\omega_{S^{(d)}}  = (\omega_S)^{(d)} = \bigoplus \limits_{i = 0}^{\infty} S_{-n+id} = S_{[k]},$$ where $k$ is the remainder of $n$ when divided by $d$.
\begin{Proposition}  
Let $S = K[x_1, \dots, x_n]$ be a standard graded polynomial ring over a perfect field $K$ of characteristic $p>0$ a non-negative integer $e$ and $d = p^eq$ a positive integer with gcd$(p,q) = 1$. Then the $p^e$-th root of canonical module $(\omega_{S^{(d)}})^{\frac{1}{p^e}}$ of the $d$-Veronese subring decomposes as a direct sum
$$(\omega_{S^{(d)}})^{\frac{1}{p^e}} = S_{[0]}^{\oplus n_0} \oplus\cdots\oplus S_{[d-1]} ^{\oplus n_{d-1}},$$ with $n_i = \frac{p^{ne}- k_e}{d}$ or $\frac{p^{ne}- k_e}{d} +1$ and $k_e$ is the remainder of $p^{ne}$ when divided by $d$.

\end{Proposition}
\noindent\textit{Proof:}  Since $\omega_{S^{(d)}}$ is given by
$$\omega_{S^{(d)}}  = (\omega_S)^{(d)} = \bigoplus \limits_{i = 0}^{\infty} S_{-n+id} = S_{[k]},$$ where $k$ is the remainder of $n$ when divided by $d$, one has that $$(\omega_{S^{(d)}})^{\frac{1}{p^e}} = (S_{[k]})^{\frac{1}{p^e}} = \bigoplus\limits_{\sum c_i \equiv_d k}   K\cdot x_1^{\frac{c_1}{p^e}}\dots x_n^{\frac{c_n}{p^e}}.$$
\noindent  Now observe that each of the $c_i$ can be written uniquely as a sum $a_ip^f + b_i$ with $0\leq b_i < p^f$. Hence $$\bigg(\sum\limits_{i=1}^n a_i\bigg)p^e + \sum \limits_{i=1}^n b_i  = \sum \limits_{i=1}^n c_i \equiv_d k.$$
Since we are interested in the asymptotic behavior of $p^e$ we can suppose that $e > f.$ Therefore by the Chinese Remainder Theorem, one has

\begin{equation}
\left( \sum a_i\right) p^e + \sum b_i \equiv_{p^f} k
    \quad \quad\text{and}\quad \quad
\left( \sum a_i \right) p^e + \sum b_i \equiv_q k 
\end{equation}
 which is
\begin{equation}
\sum b_i \equiv_{p^f} k
    \quad \quad\text{and}\quad \quad
\left(\sum a_i \right) p^e + \sum b_i \equiv_q k.
\label{cong}
\end{equation}

\noindent Let $$S_{[k,l]} = \bigoplus_{g \in G} S_g,$$ where $G$ is the set of all elements $g$ such that $g \equiv_{p^e} k$ and $g \equiv_{q} l$. Notice that this is a refinement of $S_{[h]}$ in the sense that $$S_{[h]} \cong \bigoplus \limits_{j=0}^{p^e-1} S_{[j,h]}$$ for the piece of degree congruent to $h$ module $q$. \\[0.3cm]
\noindent Let $\widetilde{\sum a_i}$ and $\widetilde{\sum b_i}$ be the congruence classes of $\sum a_i$ and $\sum b_i$ in $\mathbb{Z}/{p^e}\mathbb{Z}$ respectively; and let $\overline{\sum a_i}$ and $\overline{\sum b_i}$ be the congruence classes in $\mathbb{Z}/q\mathbb{Z}$ respectively. By \autoref{cong}, $\widetilde{\sum b_i}$ is fixed and there is no constraint on  $\widetilde{\sum a_i}$. Then, we only have to consider how $\overline{\sum a_i}$ and $\overline{\sum b_i}$ change.\\[0.3cm]
\noindent Let us fix $\overline{\sum b_i}$ and $(b_1, \dots, b_n) \in \{0, \dots, p^e-1\}^n$. Then, we have a unique $\overline{\sum a_i}$, say $\sum a_i \equiv_q k_b$, that can be reached with different vector values $(a_1, \dots, a_n)$.\\[0.3cm]
\noindent Consider the direct sum of all the $K$-modules such that $\sum a_i \equiv_q k_b$,
$$ \bigoplus \limits_{\sum a_i \equiv_q k_b} K \cdot x_1^{\frac{a_np^e + b_n}{p^e}} \dots x_n^{\frac{a_1p^e + b_1}{p^e}} \cong \bigoplus \limits_{j=0}^{p^e-1} S_{[j,k_b]} \cong S_{[k_b]}.$$
We have this direct sum for each $(b_1, \dots, b_n) \in \{0, \dots, p^e-1\}^n$, then we have to count how many vectors $(b_0, \dots, b_n)$ we have such that $\sum b_i \equiv_d \alpha$ to see how many copies of $S_{[k_b]}$ we obtain:

\noindent There are $(p^e)^n$ vectors $(b_1, \dots, b_n)$. We have to divide this total in the amount of congruence classes, that is $\frac
{p^{ne}}{d}$. But as it has to be an integer we obtain 

\begin{equation}
\frac{p^{ne}- k_e}{d}
    \quad \quad\text{or}\quad \quad
\frac{p^{ne}- k_e}{d} +1,
\end{equation}
where $k_e$ is the reminder of $p^{ne}$ divided by $d$.

\noindent Finally we obtain
$$ \bigoplus \limits_{0 \leq k_b \leq q} \left[  \bigoplus \limits_{\sum a_i \equiv_q k_b} K \cdot x_1^{\frac{a_np^e + b_n}{p^e}} \dots x_n^{\frac{a_1p^e + b_1}{p^e}} \right]^{n_i},$$
where $n_i = \frac{p^{ne}- k_e}{d}$ or $\frac{p^{ne}- k_e}{d} +1.$ Hence We can conclude that if $d = p^eq$,
$$(\omega_{S^{(d)}})^{\frac{1}{p^e}} = S_{[0]}^{\oplus n_0} \oplus\cdots\oplus S_{[d-1]} ^{\oplus n_{d-1}},$$
with $n_i = \frac{p^{ne}- k_e}{d}$ or $\frac{p^{ne}- k_e}{d} +1$ and this proves $(2)$.\qed

\section{Main Theorem}
In this section, we prove the conjecture by I. Smirnov and K. Tucker on the dual $F$-signature of the Veronese rings $S^{(d)}.$ For the sake of convenience, we restate here the conjecture.

\begin{Theorem}\label{mainthm}
Let $K$ be a perfect field of prime characteristic $p > 0$ and $S^{(d)}$ the $d$-Veronese subring of $K[x_1,\dots,x_n]$. Then, the dual $F$-signature of $S^{(d)}$ is
\begin{gather*}
s_{dual}\left(S^{(d)}\right) = \frac{1}{d}\bigg\lceil\frac{d}{n}\bigg\rceil.
\end{gather*}       
\end{Theorem}\vspace{0.1cm}

We break the proof of the theorem in several steps. Firstly, by counting the number of generators of each module, we give an upper bound $s_{dual}(S^{(d)})$.
\noindent \begin{Proposition}\label{auxlemma}
   Let $S = K[x_1, \dots, x_n]$ be a polynomial ring over a perfect field $K$ of characteristic $p>0$ and $d$ a positive integer. Consider $S^{(d)}$ the $d$-th Veronese subring of $S$. Then, 
	$$s_{dual} (S^{(d)})  \leq \frac{1}{n}\bigg\lceil \frac{n}{d}\bigg\rceil.$$
\end{Proposition}
\noindent \textit{Proof:} Recall that our goal is to find the largest $N$ such that there is surjection $$ \omega_{S^{(d)}}^{\frac{1}{p^e}} = S_{[0]}^{\oplus n_0} \oplus\cdots\oplus S_{[d-1]} ^{\oplus n_{d-1}} \twoheadrightarrow \omega_{S^{(d)}}^N$$ as $S^{(d)}$-modules. The $S^{(d)}$-linearity of such a surjection implies that the induced map $${S^{(d)}}^{\frac{1}{p^e}} = S_{[0]}^{\oplus n_0} \oplus\cdots\oplus S_{[k]}^{\oplus n_k} \twoheadrightarrow \omega_{S^{(d)}}^N = S_{[k]}^N$$
also is a surjection, where k is the remainder of n when divided by d.
			
\indent Now recall that the minimal number of generators of $S_{[i]}$ as a $S^{(d)}$-module is given by $\binom{n+i-1}{n-1}.$ Hence, by counting the minimal number of generators on each side, we have $$\binom{n+k-1}{n-1} N \leq \sum\limits_{i=0}^{k} \binom{n+i-1}{n-1}n_i,$$ which implies that  $$N \leq \frac{\sum\limits_{i=0}^{k} \binom{n+i-1}{n-1}n_i}{\binom{n+k-1}{n-1}}.$$
			
\noindent Now it is clear that $$\lim\limits_{e \to \infty} \frac{n_i}{p^{ne}} = \frac{1}{d}.$$ Hence $$s_{dual}(S^{(d)}) = \lim\limits_{e\to\infty} \frac{N}{p^{nd}} \leq \frac{\sum\limits_{i=0}^{k} \binom{n+i-1}{n-1}}{d\binom{n+k-1}{n-1}} = \frac{\binom{n+k}{n}}{d\binom{n+k-1}{n-1}}=\frac{1}{n}\bigg\lceil \frac{n}{d}\bigg\rceil.$$\qed	\\[0.3cm]
\noindent We now establish the reverse inequality by studying the possible $S^{(d)}$-linear surjections $S_{[i]}^e \twoheadrightarrow S_{[j]}^f$ for $0\leq i\leq j <d$. Notice that the $S^{(d)}$-linearity implies that $i \leq j.$ Furthermore notice that this is equivalent to giving a homogeneous map $S(i-j)^e\to S^f$ which is surjective in degree $j$. To construct such maps, we have the following proposition, which proof is postponed to the next section.\\[0.3cm]
\begin{Proposition}\label{matrix}
The homogeneous map $S(-1)^{n+k-1}\to S^{k}$ given by the matrix 
$$\Psi =\begin{bmatrix}
x_1&x_2&x_3&\dots& x_{n-1}& x_n& 0&0&\dots&0\\
0&x_1&x_2&x_3&\dots& x_{n-1}& x_n& 0&\dots &0\\
0&0& x_1&x_2&x_3&\dots& x_{n-1}& x_n& \dots &0\\
\vdots&\vdots&\vdots&\vdots&\vdots&\vdots&\vdots&\vdots&\ddots&\vdots\\
0&0&\dots&0&x_1&x_2&x_3&\dots&x_{n-1}&x_n
\end{bmatrix}$$
is surjective in degree $j\geq k$.
\end{Proposition}

When applied for $k\leq d-1$, \autoref{matrix} shows that one can construct a surjection $S_{[k-1]}^{n+k}\twoheadrightarrow S_{[k]}^{1+k}$ and the domain and target of such map has the same number of minimal generators over $S^{(d)}.$ Hence the surjection constructed in \autoref{matrix} is optimal.\\[0.3cm]
\noindent Now we are able to count on how many copies of $S_{[i]}$ we need to build a surjection $S_{[i]}^{e_i}\twoheadrightarrow S_{[k]}^{f_i}$. Obviously, if  $i=k,$  then  $e_i = f_i =1,$ while \autoref{matrix} conceives the case $e_{k-1} = n+k-1$ and $f_{k-1} = k.$		
\begin{Proposition}
  Let $0\leq i <k<d.$ Then there is a surjection 
\begin{equation*}
				S^{e_i}_{[i]}\twoheadrightarrow S^{f_i}_{[k]}
\end{equation*} such that 
\begin{gather*}
\frac{f_i}{e_i} = \frac{\binom{n+i-1}{n-1}}{\binom{n+k-1}{n-1}}.    
\end{gather*}
\end{Proposition} \begin{proof}
We prove the theorem by induction on the difference $k-i$, being the case $i=1$ already discussed. Suppose then that we have a surjection \begin{equation}\label{surj1}
			S^{e_{i+1}}_{i+1}\twoheadrightarrow S^{f_{i+1}}_{[k]}
\end{equation} such that \begin{equation}\label{binom1}
		\frac{f_{i+1}}{e_{i+1}} = \frac{\binom{n+i}{n-1}}{\binom{n+k-1}{n-1}}. 
	\end{equation}
		
		\noindent Again by \autoref{matrix}, there is a surjective map $S^{n+i}_{[i]}\twoheadrightarrow S^{i+1}_{[i+1]}$ and,  we have a surjection \begin{equation}\label{surj2}
			S_{[i]}^{e_ {i+1}(n+i)}\twoheadrightarrow S_{[i+1]}^{(i+1)e_{i+1}}
		\end{equation}
	by taking direct sums of these maps. Also, by taking direct sums of the map (\ref{surj1}), one has a surjection 
	\begin{equation}
		S^{(i+1)e_{i+1}}_{i+1}\twoheadrightarrow S^{(i+1)f_{i+1}}_{[k]}.
	\end{equation}
		The composing of the surjections (\ref{surj1} and (\ref{surj2}) yields the surjection 
	 \begin{equation}
			S_{[i]}^{e_ {i+1}(n+i)} \twoheadrightarrow S^{(i+1)f_{i+1}}_{[k]}.
	\end{equation}	
	For the last assertion of the proposition, notice that 
	\begin{equation}
		\frac{f_i}{e_i} = \frac{(i+1)f_{i+1}}{(n+1)e_{i+1}}  = \frac{i+1}{n+ i}  \frac{\binom{n+i}{n-1}}{\binom{n+k-1}{n-1}} = \frac{\binom{n+i-1}{n-1}}{\binom{n+k-1}{n-1}}. 
	\end{equation}
\end{proof} 

\indent We are now ready to prove the promised lower bound. 
\noindent \begin{Proposition}
   Let $S = K[x_1, \dots, x_n]$ be a polynomial ring over a perfect field $K$ of characteristic $p>0$ and $d$ a positive integer. Consider $S^{(d)}$ the $d$-th Veronese subring of $S$. Then, 
	$$s_{dual} (S^{(d)})  \geq \frac{1}{n}\bigg\lceil \frac{n}{d}\bigg\rceil.$$
\end{Proposition}
\begin{proof}
    Let $e$ be such that $p^ed >> e_i$ as in the previous proposition. Then, with $e_i$ copies of $S_{[i]}$ we surject on $f_i$ copies of $S_{[k]}.$ Then, if $r_{e,i}$ is the remainder of $n_{e,i}$ when divided by $e_i,$ we have that $S_{[i]}^{\oplus n_{e, i}}$ can surject in $$\frac{f_i(n_{e,i}-r_{e,i})}{e_i} = \frac{e_i (n_{e,i}-r_{e,i})}{f_i}$$
		copies of $S_{[k]}.$ Summing up in all $i$ and noticing that $\lim\limits_{e \to \infty} \frac{r_{e,i} }{p^{ne}}= 0,$ we have that 
		$$s_{dual}(S^{(d)})=\lim\limits_{e\to\infty} \frac{N}{p^{ne}} \geq \lim\limits_{e\to \infty}\sum\limits_{i=0}^k\frac{\frac{e_i (n_{e,i}-r_{e,i})}{f_i}}{p^{ne}} = \sum\limits_{i=0}^k \frac{e_i}{df_i} = \frac{\sum\limits_{i=0}^{k} \binom{n+i-1}{n-1}}{d\binom{n+k-1}{n-1}} = \frac{\binom{n+k}{n}}{d\binom{n+k-1}{n-1}}=\frac{1}{n}\bigg\lceil \frac{n}{d}\bigg\rceil.$$
\end{proof}
 \section{Appendix: An Auxiliary Lemma}

 In the proof of Theorem \ref{mainthm} we invoked as an auxiliary result Proposition \ref{matrix}, that now we give a complete proof using monomial ordering techniques. For the sake of clarity, we recall the statement that we want to prove.

\begin{Proposition}
The homogeneous map $S(-1)^{n+k-1}\to S^{k}$ given by the matrix 
$$\Psi =\begin{bmatrix}
x_1&x_2&x_3&\dots& x_{n-1}& x_n& 0&0&\dots&0\\
0&x_1&x_2&x_3&\dots& x_{n-1}& x_n& 0&\dots &0\\
0&0& x_1&x_2&x_3&\dots& x_{n-1}& x_n& \dots &0\\
\vdots&\vdots&\vdots&\vdots&\vdots&\vdots&\vdots&\vdots&\ddots&\vdots\\
0&0&\dots&0&x_1&x_2&x_3&\dots&x_{n-1}&x_n
\end{bmatrix}$$
is surjective in degree $j\geq k$
\end{Proposition}		

The proposition follows immediately if we prove that $I_{k}(\Psi) = (x_1,\dots,x_n)^{k}$. Indeed, one always has
$$ (x_1,\dots,x_n)^{k} = \Fitt_0(\coker (\Psi)) \subset \Ann(\coker (\Psi))$$ and this implies that $\coker(\Psi)_j = 0$ for $j\geq k.$ Hence, our focus will be on proving that $I_{k}(\Psi) = (x_1,\dots,x_n)^{k}$. 

We set some notations and conventions. Equip the polynomial ring $S$ with lexicographical monomial order $x_1<\cdots < x_n$, and given $r \geq 1$, consider the $r\times (n+r-1)$ matrix
\begin{gather*}
    M(n,r) := \begin{bmatrix}
				x_1&x_2&x_3&\dots& x_{n-1}& x_n& 0&0&\dots&0\\
				0&x_1&x_2&x_3&\dots& x_{n-1}& x_n& 0&\dots &0\\
				0&0& x_1&x_2&x_3&\dots& x_{n-1}& x_n& \dots &0\\
				\vdots&\vdots&\vdots&\vdots&\vdots&\vdots&\vdots&\vdots&\ddots&\vdots\\
				0&0&\dots&0&x_1&x_2&x_3&\dots&x_{n-1}&x_n
			\end{bmatrix}.
\end{gather*}
Next, for each $\alpha_1,\dots,\alpha_n$ non-negative integers such that $\sum_{k=1}^n\alpha_k = r$, consider the $r\times r$ matrix $M_{\alpha_1,\dots,\alpha_n}(r)$ constructed as follows:
\begin{itemize}
    \item For $1 \leq j \leq \alpha_1$, the $j$-th column of $M_{\alpha_1,\dots,\alpha_n}(r)$ is the column of $M(n,r)$ for which $x_1$ appears on the $j$-th row;
    \item For $\big(\sum_{i=1}^k\alpha_i    \big)+1 \leq j \leq \sum_{i=1}^{k+1}\alpha_i$, the $j$-th column of $M_{\alpha_1,\dots,\alpha_n}(r)$ is the column of $M(n,r)$ for which $x_{k+1}$ appears on the $j$-th row.
\end{itemize}
\noindent In order to illustrate this construction, it follows an example. 
\begin{Example}
{\normalfont 
    Let $S = K[x_1,x_2,x_3]$ and $r = 6$ be chosen and consider $\alpha_1 = 2$ , $\alpha_2 = 3$ and $\alpha_3 = 1$. Then one has 
    \begin{gather*} 
    M(3,6)= \begin{bmatrix} 
x_1 & x_2  & x_3  & 0 & 0 & 0 & 0 & 0 \\ 
0 & x_1 & x_2 &  x_3 & 0 & 0 & 0  & 0 \\ 
0 & 0 & x_1 & x_2 & x_3 & 0 & 0 & 0 \\ 
0 & 0 & 0 & x_1 & x_2 & x_3 & 0 & 0\\ 
0 & 0 & 0 & 0 & x_1 & x_2 & x_3 & 0\\ 
0 & 0 & 0 &0  &0  & x_1 & x_2 & x_3
\end{bmatrix} \:\:\:\: \textrm{and} \:\:\:\: M_{3,2,1}(6) = \begin{bmatrix}
x_1 & x_2 & x_3 & 0 & 0  &0  \\ 
0 & x_1 & x_2 & 0 & 0 &0 \\ 
0 & 0 & x_1  & x_3 & 0 & 0\\ 
0 & 0 & 0 & x_2 & x_3 &0 \\ 
0 & 0 & 0 & x_1 & x_2 & 0 \\ 
0 & 0 & 0 & 0 & x_1 & x_3 
\end{bmatrix}  .
    \end{gather*}
}
\end{Example}
\noindent Notice that each $r\times r$ minor of $M(n,r)$ is determinant of $M_{\alpha_1,\dots,\alpha_n}(r)$ for some non-negative integers $\alpha_1,\dots,\alpha_n$ such that $\sum_{k=1}^n\alpha_k = r$. In the next proposition, denote
\begin{gather*}
    \mathfrak{M}(n,r) = \bigg\{(x_1,\dots,x_n) \in \mathbb{N}_0^n\:\:;\:\:\sum_{k=1}^nx_k = r\bigg\}.
\end{gather*}
Note that $x_1^r$ is the minimum element of $\mathfrak{M}(n,r)$.

We are now ready to prove the desired equality.
\begin{Proposition}
    Let $S = K[x_1,\dots,x_n]$ be the polynomial ring in $n$ indeterminates over a field $K$ equipped with lexicographical monomial ordering $x_1 < \cdots < x_n$ and $r$ a positive integer. Then
    \begin{gather*}
        I_n\big(M(n,r)\big) = (x_1,\dots,x_r)^r
    \end{gather*} 
\end{Proposition}
\noindent\textit{Proof:}

Before starting the actual proof we first describe loosely what is the idea behind it. In the previous lines, we noticed that $x_1^r$ belongs to the ideal $I_n\big(M(n,r)\big)$. Our main goal is to prove that if all monomials $m'$ that are smaller than a given monomial $m$ belongs to $I_n\big(M(n,r)\big),$ then $m$ also belongs to $I_n\big(M(n,r)\big).$ We prove this by giving an explicit $r\times r$ minor of $M(n,r)$ consisting of a combination of $m$ and smaller monomials. The result then follows by induction.
\\
Now we proceed with the proof. It is clear that $I_n\big(M(n,r)\big) \subseteq (x_1,\dots,x_r)^r$.  In order to prove the other inclusion, it is enough to show that $x_1^{\alpha_1}\cdots\: x_n^{\alpha_n} \in I_r\big (M(n,r)\big )$ for any $n$-tuple $(\alpha_1,\dots,\alpha_n)$ of non-negative integers with $\sum_{k=1}^n\alpha_k = r$. We'll prove the following claim:\\ 
\begin{center}
Let $m = x_1^{\alpha_1}\cdots\: x_n^{\alpha_n}$ a monomial with $\sum_{k=1}^n\alpha_k = r$ be chosen. If $m' \in I_r\big (M(n,r)\big )$ for all $m' < m$, then $m \in I_r\big (M(n,r)\big )$.    
\end{center}

\noindent As mentioned before $x_1^r \in I_r\big (M(n,r)\big )$. Indeed, if $\alpha_1 = r$ and $\alpha_i = 0$ for all $1 < i \leq n$, one has 
\begin{gather*}
    \det\big( M_{r,0,\dots,0}\big) = \det\begin{bmatrix}
x_1 &x_2  & \cdots  & x_{r-1}  & x_r \\ 
0 & x_1 & \cdots & x_{r-2} & x_{r-1}\\ 
0 & 0 & \ddots & \vdots & \vdots \\ 
0 & 0 &\cdots  & x_1 & x_2 \\ 
0 & 0 & \cdots  & 0 & x_1
\end{bmatrix} = x_1^r.
\end{gather*}
Next let $m = x_1^{\alpha_1}\cdots\: x_n^{\alpha_n}$ be a monomial with $\sum_{k=1}^n\alpha_k = r$. Now let $m$ be a degree $r$ monomial and suppose that for any $m'< m$ and that $m' \in I_r\big (M(n,r)\big )$. Observe that $$\det\big(M_{\alpha_1,\alpha_2,\dots,\alpha_r}(r)\big) = x_1^{\alpha_1}\det\big(M_{0,\alpha_2,\dots,\alpha_r}(r-\alpha_1)\big)$$
and that
\begin{gather*}
    \det\big(M_{0,\alpha_2,\dots,\alpha_r}(r - \alpha_1)\big) = x_2\det\big(M_{0,\alpha_2-1,\dots,\alpha_r}(r - \alpha_1-1)\big) - x_1f(x_1,\dots,x_ n),
\end{gather*}
where $f(x_1,\dots,x_ n)$ is a homogeneous polynomial in $S$ with degree $r-(\alpha_1 +1)$. Hence
\begin{gather*}
    \det\big(M_{\alpha_1,\alpha_2,\dots,\alpha_r}(r)\big) = x_1^{\alpha_1}x_2\det\big(M_{0,\alpha_2-1,\dots,\alpha_r}(r - \alpha_1-1)\big) - x_1^{\alpha_1+1}f(x_1,\dots,x_ n)
\end{gather*}
Since each monomial of $x_1^{\alpha_1+1}f(x_1,\dots,x_ n)$ is smaller than $m$ and $\det\big(M_{\alpha_1,\alpha_2,\dots,\alpha_r}(r)\big)$ is a minor of $M(n,r)$, by induction hypothesis, one concludes that $$x_1^{\alpha_1}x_2\det\big(M_{0,\alpha_2-1,\dots,\alpha_r}(r - \alpha_1-1)\big) \in I_r(M(n,r)).$$ Suppose that we have proved that $x_1^{\alpha_1}x_2^i\det\big(M_{0,\alpha_2-i,\dots,\alpha_r}(r - \alpha_1-i)\big) \in I_r(M(n,r))$ for all $1\leq i < \alpha_2 $. Again we have    
\begin{gather*}
    x_1^{\alpha_1}x_2^i\det\big(M_{0,\alpha_2-i,\dots,\alpha_r}(r - \alpha_1 -i)\big) = x_1^{\alpha_1}x_2^i\bigg(x_2\det\big(M_{0,\alpha_2-(i+1),\dots,\alpha_r}(r - \alpha_1 - (i + 1))\big) - x_1f'(x)\bigg) \\= x_1^{\alpha_1}x_2^{i+1}\det\big(M_{0,\alpha_2-(i+1),\dots,\alpha_r}(r - \alpha_1 - (i + 1))\big)  - x_1^{\alpha_1+1}x_2^{i}f'(x),
\end{gather*}
where $f'(x_1,\dots,x_ n)$ is a homogeneous polynomial in $S$ with degree $r-(\alpha_1 +i +1)$. Similarly, since each monomial of $x_1^{\alpha_1+1}x_if'(x_1,\dots,x_ n)$ is smaller than $m$ and $$x_1^{\alpha_1}x_2^i\det\big(M_{0,\alpha_2-i,\dots,\alpha_r}\big)(r - \alpha_1 - i) \in I_r(M(n,r)),$$ by induction hypothesis, one concludes that $$x_1^{\alpha_1}x_2^{i+1}\det\big(M_{0,\alpha_2-(i+1),\dots,\alpha_r}(r - \alpha_1 - (i + 1))\big) \in I_r(M(n,r)).$$ Proceeding until $i = \alpha_2-1$ and repeating the argument, one gets  
\begin{gather*}
    x_1^{\alpha_1}x_2^{\alpha_2}\det\big(M_{0,0,\alpha_3\dots,\alpha_r}(r - \alpha_1 - \alpha_2)\big) \in I_r(M(n,r).
\end{gather*}
In general, let $1 \leq i \leq n$  and $0\leq e < \alpha_i$. Setting $\tau = r - \sum_{k=1}^i\alpha_k$, note that 
\begin{gather*}
    \det\big(M_{0,\dots,0,\alpha_i - e,\dots,\alpha_r}(\tau - e)\big) = x_i \det\big(M_{0,\dots,0,\alpha_i - (e+1),\dots,\alpha_r}(\tau -(e +1))\big) + \sum_{j=1}^{i-1}(-1)^jx_{i-j}\det(S_j),
\end{gather*}
where $S_j$ is the submatrix $M_{0,\dots,0,\alpha_i - e,\dots,\alpha_r}(\tau - e)$ obtained by omission of the first column and $j$ row. Thus    
\begin{gather*}
    x_1^{\alpha_1}\cdots\: x_i^{e}\det\big(M_{0,\dots,0,\alpha_i - e,\dots,\alpha_r}(\tau - e)\big) = x_1^{\alpha_1}\cdots\: x_i^{e+1} \det\big(M_{0,\dots,0,\alpha_i - (e+1),\dots,\alpha_r}(\tau - (e+1))\big) \\ + \sum_{j=1}^{i-1}(-1)^j x_1^{\alpha_1}\cdots \:x_{i-j}^{\alpha_{i-j}+1}\cdots \:x_i^{e}\det(S_j).
\end{gather*}
Supposing that $x_1^{\alpha_1}\cdots\: x_i^{e}\det\big(M_{0,\dots,0,\alpha_i - e,\dots,\alpha_r}(\tau -e)\big) \in I_r\big(M(n,r)\big)$, by induction hypothesis, one concludes that $x_1^{\alpha_1}\cdots\: x_i^{e+1} \det\big(M_{0,\dots,0,\alpha_i - (e+1),\dots,\alpha_r}(\tau - (e+1))\big) \in I_r\big(M(n,r)\big)$. If we repeat this process until $x_1^{\alpha_1}\cdots\: x_r^{\alpha_r -1} \det\big(M_{0,\dots,0,1}(1)\big)$ and argue as above, we conclude that $x_1^{\alpha_1}\cdots\: x_r^{\alpha_r} \in I_r\big(M(n,r)\big)$.
\qed\\[0.5cm]
\noindent \textbf{Acknowledgment:} This question was proposed to us during the conference PRAGMATIC 2023 at Catania (Italy), where Luis Núñez-Betancourt and Eamon Quinlan-Gallego were part of the mentoring team. We let our gratitude to the PRAGMATIC 2023 organization, as well as Luis Núñez-Betancourt and Eamon Quinlan-Gallego for their valuable contributions with discussions and insightful comments. Additionally, we would like to thank CAPES for funding the transportation expenses of one of the team members.

\newpage
\bibliographystyle{alpha}
\bibliography{References} 
\end{document}